\documentclass[11pt,reqno]{amsproc}
\usepackage{amstext,amsmath,amssymb,amsthm, bm}
\usepackage{subfig}
\usepackage[english]{babel}
\usepackage{graphicx}
\usepackage{tabularx}
\usepackage{enumerate, hyperref}

\usepackage{tikz}
\usetikzlibrary{shapes.geometric,calc}

\def\({\Big(}
\def\){\Big)}

\renewcommand{\t}{\tilde}

\newcommand{\Z}{  \mathbb Z }
\newcommand{\N}{  \mathbb N }
\newcommand{\Q}{  \mathbb Q }

\newcommand{\be}  { \mbox{$\boldsymbol{\epsilon}$}}
\newcommand{\bo}  {\mbox{$\boldsymbol{\omega}$}}
\newcommand{\bx}  {{\bf x}}
\newcommand{\by}  {{\bf y}}
\newcommand{\ov}{\overline}

\renewcommand{\t}{\tilde}

\newtheorem{Thm}{Theorem}[section]
\newtheorem{Lemma}[Thm]{Lemma}

\newtheorem{Prop}[Thm]{Proposition}
\newcommand{\dsize}{\displaystyle}
\newcommand{\cal}{\mathcal}
\numberwithin{equation}{section}
\newcommand{\ld}[1]{{\color{red}{#1}}}

\author{L. De Carli}
\address{Laura De Carli, Dept. of Mathematics, Florida International University,   Miami, FL 33199, USA.}
\email{decarlil@fiu.edu}

\author{A. Echezabal}
\address{Andrew Echezabal, Dept. of Mathematics, Florida International University,   Miami, FL 33199, USA.}
\email{ ancheza@fiu.edu}

\author{I. Morell}
\address{Ismael Morell, Dept. of Mathematics, Florida International University,   Miami, FL 33199, USA.}
\email{imore040@fiu.edu  }
\begin{document}
	

\title [Egyptian fractions  meet the Sierpinski  triangle] {{\Large  Egyptian fractions meet the Sierpinski  triangle   }}

\maketitle






\begin{abstract}
	 We explore a novel link between two seemingly disparate mathematical concepts: Egyptian fractions and fractals. By examining the decomposition of rationals into sums of distinct unit fractions, a practice rooted in ancient Egyptian mathematics, and   the  arithmetic operations  that can be performed using this decomposition, we uncover fractal structures that emerge from   these representations.  
	\end{abstract}

\section{ Egyptian fractions}
  
The ancient Egyptians wrote    fractions  as   a formal
 sum of distinct   unit fractions  $\frac{1}{n}$, with $n\ge 2$.   For example,   they wrote  $\frac 25 $ as $\frac 13+\frac {1}{15}$ or as $\frac 14+\frac{1}{12}+\frac 1{15}$ but  not  as $  \frac 15+\frac  15$.
 Unit fractions  were represented with an almond-shaped  symbol (the "mouth")  over the hieroglyphic symbol for  the number  which indicated the denominator. \footnote{ There were some exceptions to this rule: for example,   the Egyptian had special symbols for the fractions $\frac 12$ and $\frac 23$.} 
 
%
 This peculiar representation of fraction lead to  define 
 {\it Egyptian fractions} as   formal sums of  distinct unit fractions.  
 Egyptian fractions have captivated mathematicians for centuries due to their unique properties and historical significance and are still an active topic of  research in number theory. See e.g. \cite {V},  the recent \cite{E} and the references cited there.

An efficient algorithm 
 first  introduced  by   Fibonacci in 1202 in  the {\it Liber Abaci}    can be used to express any given proper fraction $\frac pq$ as  an  Egyptian fraction with at most $p$ terms.  
 For  instance,   the Fibonacci algorithm  gives
\begin{equation}\label{1-1}
  \frac 2{2k+1}= \frac 1{ k+1}+\frac 1{(k+1)(2k+1)} 
\end{equation}
for every   integer  $k>0$. For a detailed description and applications of the Fibonacci greedy algorithm see e.g.  \cite{Cr},   \cite{L}.

Natural numbers can be represented  as a Egyptian fractions too: for instance,
 $1=\frac 12+\frac 13+ \frac{1}{12}+\frac{1}{18}+\frac 1{36}.$ 
Repeated applications of   this identity  and the formula  \eqref{1-1}  allow to represent any positive integer  as  a sum of unit fractions. However, these representations can be extremely lengthy  and computationally complex.
 \begin{figure}
 	\begin{tikzpicture}[scale=0.5]
 		%
 		\fill[black] (0,0) ellipse (.8 and 0.4); 
 		\foreach \x in {-0.5, 0, 0.5} {
 			\draw[black, line width=3pt] (\x,-0.5) -- (\x,-1.5); 
 		}
 		\node[below, black] at (0,-1.6) {$\frac{1}{3}$};
 		
 		\fill[black] (3,0) ellipse (.8 and 0.4); 
 		\foreach \x in {-0.6, -0.2, 0.2, 0.6} {
 			\draw[black, line width=3pt] (3+\x,-0.5) -- (3+\x,-1.5); 
 		}
 		\node[below, black] at (3,-1.6) {$\frac{1}{4}$};
 		
 		\fill[black] (6,0) ellipse (.8 and 0.4); 
 		\draw[black, line width=3pt] (6.5,-1.1) arc[start angle=0, end angle=180, radius=0.5]; 
 		\node[below, black] at (6,-1.6) {$\frac{1}{10}$};
 	\end{tikzpicture}\caption{ Hieroglyphic  for Egyptian fractions}
 \end{figure}
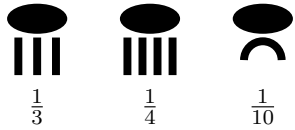

 The sum of two Egyptian fractions can be expressed as a formal sum of terms of the form $\frac an$, where $a=1,\,2$.  When $n$ is even, terms of the form $\frac 2n$  can be   simplified  and when $n$ is odd they   can be   rewritten as a sum of unit fractions using  the formula \eqref{1-1}.   
The ancient Egyptian scribes used tables to write   $\frac 2n$   as sums of unit fractions; one of these tables  is part of the Rhind Mathematical Papyrus (circa 1650 BC),  which is now housed in the British Museum.

Also the difference of Egyptian fractions is straightforward, as long as  terms in the form of  $\frac{a}{n}$, with $a=   \pm 1,  $ can be considered in the result.  For example,  the difference of 
$\frac 15+ \frac 1{10}+\frac 1{20}$ and 
$ \frac 1{10}+\frac 1{30}$  (which are  Egyptian fractions that represent $\frac 7{20}$ and $\frac2{15}$) can be written as $ \frac 15+\frac 1{20}-\frac 1{30}$. 

 \medskip
 In this paper we identify    standard  Egyptian fractions    with  unique sets of the form   $X  =\{\frac{x_j}{j+1}\}_{j\in\N}$,  where  $\N=\{1, \, 2,\, ...\}$ is the set of positive integers,   $x_j\in \{  0,1\}$   and  only finitely many $x_j$ are nonzero.
Different  sets correspond to different Egyptian fractions even when   the terms of the sets add up to  the same number.
 
 We also  define the   {\it signed Egyptian fractions} as formal  sums and differences of  unit fractions,
 and we identify   them with  unique  sets of the form  $X  =\{\frac{x_j}{j+1}\} _{j\in\N} $,  where $x_j\in \{-1,  0,1\}$   and  only finitely many $x_j$ are nonzero.
When the context makes it clear and no distinction is required, we will use the term "Egyptian fractions" to refer to both standard and signed Egyptian fractions.

 We say that two Egyptian fractions  $X  =\{\frac{x_j}{j+1}\} _{j\in\N} $ and $Y  =\{\frac{y_j}{j+1}\} _{j\in\N} $   are {\it equivalent} if their terms  add up to the same number and   they are
  {\it disjoint} if    $x_j y_j= 0$ for every $j$.   For instance, the Egyptian fraction $\frac 13+\frac 14$ and $\frac 12-\frac 15$ are disjoint, but $\frac 13+\frac 14$ and $-\frac 13+\frac 15$ are not.
   
Performing arithmetic operations on disjoint Egyptian fractions is straightforward. Specifically, the sum of two disjoint Egyptian fractions   identified by the sets   $X$ and $Y$   by  $X\cup Y$  and their difference   by  $X\cup (-Y)$, where   $-Y  =\{\frac{-y_j}{j+1}\} _{j\in\N} $.

%
It can be proved that two Egyptian fractions have  equivalent disjoint Egyptian fraction representations, but we will  not provide a proof here.

\section{Egyptian fractions and    decimals in base two or three}
 
 Our  system of numbers is in base ten;  each position in the number represents a   power of ten, and the value of the number is determined by summing these weighted positions.  
  
More generally, any base $m \ge 2$  and any set of $m$ symbol, each corresponding to  one of the integers $0,\, ...,\, m-1$, can be used to represent numbers. 
 In base $m$, the  symbol $[a_k... a_{1}  a_0.\, a_{-1} a_{-2}...]_m $   denotes the number 
 $$a= a_km^k+...+ a_1 m+a_0+\frac{a_{-1}}{m}+ \frac{a_{-2}}{m^2}+...
 .$$
 Decimal representations  of a number in any base  may  not be unique; for instance,    $1=0.9999...$  in base ten.  In base two,   the fraction  $\frac 12$  is  represented by the decimal  $[0.100...]_2$ and also by    $[0.01111...]_2$.
 
 We say that a number   $x\in [0,1)$  has   a {\it finite decimal representation}  in a given  base if such representation  contains only finitely many nonzero digits.     
 
 It is not difficult to verify that   a   number cannot have   two different finite decimal representations in a given basis, but  the examples above show that it  can have    a finite  and  an infinite  decimal representations. 
   
  When multiplying or dividing   a decimal number by a power of  the base, the position of the decimal point shifts either to the right (when multiplying) or to the left (when dividing). For instance,    $m^2* [0.\,x_1x_2x_3 ...]_m ...= [x_1x_2.\,x_3 ... ... ]_m$ and $\frac 1{m^2}* [0.x_1x_2x_3...]_m ...= [0.00x_1x_2 x_3  ... ]_m$.
   
   \medskip

    The basis $m=2$ is especially relevant in the  applications; we can write  every non-negative real number   as $x=x_k2^k+...+ 2x_1 +x_0+\frac{x_{-1}}{2}+ \frac{x_{-2}}{2^2}+...
  $  where each  $x_j$  is   either $0$ or $1$, and we can represent $x$   by the binary number   $  [x_k.... x_1x_0.\, x_{-1}x_{-2}...]_2 $.

 \medskip
Numbers can also be represented using  the so-called {\it balanced ternary} basis. This is   a numerical system  in base $3$  but with   digits $\{0,\, 1,\, -1\}$.   
 
 The balanced ternary representation  offers computational  advantages over the standard  base 3 representation, such as simpler rules for performing  arithmetic operations. 
 It is   used in   some analog circuit design   where signals have three states (positive, zero, and negative) and in   modern machine learning algorithms. See e.g. \cite{FT}, \cite{TT2} and the references cited  there.
 	
 Each real number  $\omega$ can be written as 
 $[\omega_k ...\omega _1  \omega _0.\,\ \omega _{-1} \omega _{-2} ... ]_3=  \omega_k3^k+...+ 3\omega _1 +\omega _0+\frac{\omega _{-1}}{3}+ \frac{\omega _{-2}}{3^2}+...
 $  where  the coefficients   $\omega _j$  are   either $0$,  $1$ or $-1$.  
 
 When $\omega$   is a non-negative integer, we can find its  balanced ternary representation as follows. Start by setting  $ n_0=\omega $  and   let   $r_0$ be the reminder of the division of $n_0 $ by $3$.  If  $r_0=0$ or $ r_0=1$, set  $\omega _0=r_0$ and if  $r_0=2$, set  $\omega _0=-1$. Then, set  $  n_{1} = (n_0-\omega_0)/3 $,   replace $n_0$ with $n_1$   and     repeat the process until   $n_k=0$.  

   For instance,  the number 8 can be represented  in balanced ternary   using the following steps:
    \begin{itemize}
    \item 	$n_0=8 = 3*2+2$, so $\omega _0= -1,\ n_1= (8-(-1))/3=3$
    \item $n_1= 3*1+0$, so $\omega _1=0,\ n_2= (3-(0))/3=1$
    \item  $n_2=3*0+1$, so $\omega _2= 1,\ n_3= 0$
     
    \end{itemize}
If we let  $-1=\ov 1$, we can write $8=[1 0\ov1 ]_3$.

\medskip


We denote with  $\Z_2=\{0,1\}$    the  set of integers modulo $2$   and  with $\Z_3=\{0, 1, -1\}$  the set  of integers modulo $3$.  Addition and multiplication in   $\Z_2$ and $\Z_3$ are   defined in the usual way, but with results taken modulo $2$, or in balanced ternary.   The sets  $\Z_2 $ and $\Z_3$  form    commutative rings; they   are also fields  because every nonzero element has a multiplicative inverse.

\medskip
We define $\ell^1(\Z_2)$ as the set of vectors 
$\be =( \epsilon_1,\, ...,\, \epsilon_n\, ...)$ with   $\epsilon_j\in\{0,1\}$  and  $\sum_{k=1}^\infty\epsilon_k <\infty$.  
Similarly, we define  $\ell^1(\Z_3)$ as the set of vectors 
$\bo =( \omega_1,\, ...,\, \omega_n\, ...)$ with   $\omega_j\in\{0,1,\, -1\}$  and  $\sum_{k=1}^\infty|\omega _k| <\infty$.
These definitions imply that only  finitely many entries  $\epsilon_j$ and $\omega_j$ are nonzero. 

It is not too difficult to verify that  the sets 
 $\ell^1(\Z_2)$ and $\ell^1(\Z_3)$  are vector spaces  on $\Z_2$ and $\Z_3$, respectively. In both cases, scalar multiplication is performed componentwise, and vector addition is carried out componentwise in base $2$ (or in balanced base $3$).
  
   For example, given  $\bo= (1, -1, 1, 1, 0,...0 ...) $, $\bo'= (0,  1, 1, 1, 0,...0 ...) $    in $\ell^1(\Z^3)$,  we have
   $\bo+\bo'=( 1, 0, -1, -1, 0, ...0...)$ and $(-1) \bo= (-1,  1, -1, -1, 0,...0 ...)$.

 \medskip
 
 In the previous section we have identified     signed Egyptian fractions   with    sets  $ X= \left\{  \frac{x_j}{j+1}\right\}_{j\in\N} $, where  $
 x_j\in \{-1, 0, 1\}$.   When all $ x_j$ are non-negative,   $X$ is a standard Egyptian fraction.
 
We let ${\cal E}$  be the set of  the standard  Egyptian fractions,   and  we define the map 
   $$\mbox{$h :\ell^1(\Z^2)\to {\cal E}$,  \quad 
   $ h ( \be) =\left\{\frac{\epsilon_j}{j+1}\right\} _{j\in\N} . $}
$$
   Similarly, we let $\t {\cal E}$ be the set of   the signed Egyptian fractions  
 and we define the   map
   $$\mbox{$  h :\ell^2(\Z^3)\to \t{\cal E}$, \quad   
   $   h ( \bo) =\left\{\frac{\omega_j}{j+1}\right\}_{j=1}^\infty $.}$$
  These  mappings   establish a natural correspondence  between  standard (or signed) Egyptian fractions and  $\ell^2(\Z^2)$ (or $\ell^2(\Z^3)$).  
 
 \medskip
 We  let $\Sigma: \t {\cal E}  \to \Q$ be the map that associates to an  Egyptian fraction (standard or signed)  the sum of its term. That is,
 $$\Sigma\big(\big\{\frac{x_j}{j+1}\big\}_{j\in\N}  \big)=  \sum_{j=1}^\infty\frac{x_j}{j+1}.$$


 The  mappings  $\Sigma  h (\be)= \Sigma(h(\be))$    map  $\ell^1(\Z^2)$  and  $\ell^1(\Z^3)$ onto $\Q$.   Two equivalent Egyptian fractions  have the same  image.
 
  It  is natural to ask   whether the map $\Sigma h$   is linear in  $\ell^1(\Z_2)$ or $\ell^1(\Z_3)$.    Specifically, we ask:   
 
  \medskip
\noindent{\bf Q1:}   For which $\bx, \, \by\in\ell^1(\Z_2)$ we have  $ \Sigma   h(\bx+\by)= \Sigma   h(\bx)+\Sigma   h(\by) $ ?
  
  \medskip
\noindent{\bf Q2:}  For which $\bx, \, \by\in\ell^1(\Z_3)$ we have  $ \Sigma   h(\bx\pm \by)= \Sigma   h(\bx)\pm\Sigma    h(\by) $?

  \medskip
To best answer  these questions we introduce  the following notation:
   given  the vectors $\bx, \, \by$ in $\ell^1(\Z_2)$ or in $\ell^1(\Z_3)$, we  let  ${\bf z} ={\bf  z(x,y)} $ be the vector  for which  $z_j=1$ when   $x_j =y_j=1$, $z_j=-1$ when   $x_j =y_j=-1$ and $z_j=0$  for all other $j$.   
   
    When $\bx, \, \by\in \ell^1(\Z_2)$, the vector   ${\bf z (\bx, \by)}$ is zero  if and only  $x_jy_j=0$   for each index $j$, which implies that the Egyptian fractions     $  h(\bx')$ and $  h(\by')$ are disjoint.  If $\bx, \, \by\in \ell^1(\Z_3)$,  the vector   ${\bf z (\bx, \by)}$ is zero  if and only  either $x_jy_j=0$   or $x_j+y_j=0$ for each index $j$.
  

  \medskip
  The following proposition answers {\bf Q1}.
  
  \begin{Prop}\label{P-Sum-2}  
 
  	Let $\bx,\, \by\in \ell^1(\Z_2)$.  Then, $   \Sigma h(\bx) +   \Sigma h( \by) =   \Sigma h(\bx +\by)$ if and only if   ${\bf z}=0$.  
  	\end{Prop}

 \begin{proof}   If ${\bf z}=0$, the Egyptian fractions  $h(\bx)$ and $h(\by)$ are disjoint;  their sum by is represented by   the union of the  sets associated to them, and  
 $ \Sigma h(\bx+\by)   =\Sigma h(\bx)+\Sigma h(\by)$.

 \medskip
 	Conversely, we assume that   $\Sigma h(\bx+\by) =\Sigma h(\bx)+\Sigma h(\by) $;  letting  $\bx'=\bx -{\bf z}$  and $\by'=\by- {\bf z}$ and recalling that addition in $\ell^2(\Z_2)$ follows binary rules, we  have that   $\bx+\by =	 \bx' + \by' +2{\bf z}= \bx' + \by'$. 
 	   The Egyptian fractions $h(\bx') $ and $ h(\by')  $ are disjoint,   and by the first part of the proof, 
 	    $$ 
 	   \Sigma h(\bx+\by)=	\Sigma h(\bx'+\by')= \Sigma h(\bx') + \Sigma h(\by')  .$$
 But since  we also have
 	   $   \Sigma  h(\bx+\by)=\Sigma h(\bx) + \Sigma h(\by)=   \Sigma h  (\bx') +  \Sigma h  (\by') + 2\Sigma h({\bf z});  
 	   $ 
 	  necessarily  $2\Sigma h({\bf z})=0$,    and so  ${\bf z=0}$. 
 	  \end{proof}
 	
 	\medskip
 	 The answer to {\bf Q2} is slightly more complicated.
 	
 	\begin{Prop}\label{P-Sum-3}  
 		a)   Let $\bx ,\, \by\in \ell^1(\Z_3)$.  If      ${\bf z  }=0$, then   $ \Sigma   h(\bx) + \Sigma   h( \by) = \Sigma   h(\bx +\by)$.
 		
 		b)  If $ \Sigma  h(\bx) + \Sigma   h( \by) = \Sigma   h(\bx +\by)$, 
 		then   $\Sigma h( {\bf z} )=0$.
 	\end{Prop}

\begin{proof}    
	  If ${\bf z(x,y)}=0$, then  either $x_jy_j= 0$ or $x_j+y_j=0$ for every index $j$. If $x_jy_j=0$ for every $j$,  the Egyptian fractions  $h(\bx)$ and $h(\by)$ are disjoint  and 
    $ \Sigma h  (\bx+\by) =    \Sigma h  (\bx )+\Sigma h  (\by ). $  
     If   the set $I$   of indices $j$ for which $x_jy_j\ne 0$ and $x_j+y_j=0$   is nonempty, we   have that
    $$\sum_{j\in I} \frac {x_j}{j+1}+ \sum_{j\in I} \frac {y_j}{j+1} = \sum_{j \in I} \frac {x_j+y_j}{j+1}=0,$$
    and  so
    $$ \Sigma h  (\bx+\by) =  \sum_{j\in \N-I} \frac {x_j+y_j}{j+1}  = \sum_{j\in \N-I} \frac {x_j }{j+1}+\sum_{j\in \N-I} \frac { y_j}{j+1}=\Sigma h  (\bx )+\Sigma h  (\by )  $$
     as required.
   
Let us prove  b). 
We let  $\bx'=\bx -{\bf z}$  and $\by'=\by- {\bf z}$ and  we  observe that  the components of $\bx'$ and $\by'$ satisfy  either $x_j'y_j'=0$ or $x_j'+y_j'=0$.   Furthermore,    $\bx'+\by'$ and ${\bf z}$ are disjoint.

The operations  in $\ell^2(\Z_3)$  follows  balanced ternary rules,  and so  $2{\bf z}=- {\bf z}$.    Thus,
$ \bx +\by =  \bx'+\by' +2{\bf z}=\bx'+\by' -{\bf z} $ and  by the previous observations and  the first part of the proof  
\begin{equation}\label{1}
\Sigma h  (\bx +\by ) =\Sigma h  (\bx'+\by'+2{\bf z})=\Sigma h  (\bx'+\by'- {\bf z})= \Sigma h  (\bx'+\by')-\Sigma h  ({\bf z}).
\end{equation}
We also have
\begin{equation}\label{2}
\Sigma  h  (\bx)+\Sigma h  (\by)   = \Sigma  h  (\bx' + {\bf z})+\Sigma h  (\by'+  {\bf z}) $$$$=	\Sigma  h  (\bx') + \Sigma h  (\by')  + 2 \Sigma h  ({\bf  z})= \Sigma h  (\bx'+\by') + 2 \Sigma h  ({\bf  z}).
\end{equation}
By assumption,  $\Sigma  h  (\bx)+\Sigma h  (\by)=\Sigma  h  (\bx+ \by)$   and so from \eqref{1} and \eqref{2}  follows that 
  $-\Sigma h  ({\bf z})=2  \Sigma h  ({\bf z})$, and    so
$ \Sigma h  ({\bf z})=0$.
 
\end{proof}

\section{Egyptian fractions  and fractals}

The   Sierpinski triangle  is the set obtained  from a recursive process of  removing  smaller and smaller triangles from an initial   triangle.     It is named after the Polish mathematician Waclaw Sierpinski (1882 -1969)  but appeared as a decorative pattern many centuries before his work. 

In this paper,    the initial triangle $T_0$ is   the closed  triangle with vertices  $(0, 0)$, $(0,1)$ and $(1,0)$,
 \begin{equation} \label{e-To}
 T_0=\{(x, y)\in [0,1]\times [0,1]  \ : \ 0 \leq x+y\leq 1\}.
 \end{equation}

\begin{figure}[h]
	\includegraphics[width=0.4\textwidth]{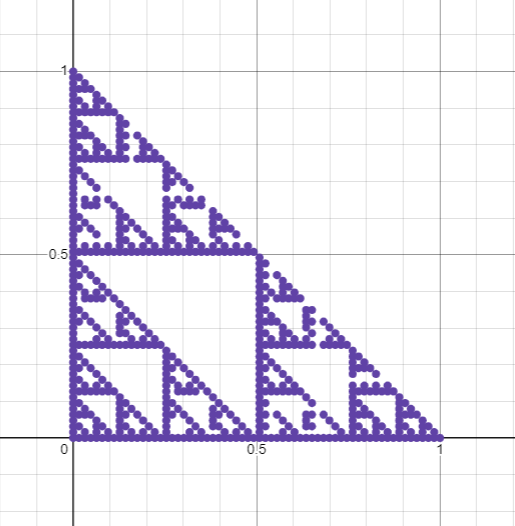}
	\caption{\small{ The Sierpinski triangle}}
	\label{fig:triangle}
\end{figure}

To formally define the Sierpinski triangle, we   define recursively a decreasing sequence of sets $\{S_i\}_{i = 0}^{\infty}$, where  $S_0=T_0$ and each $S_i$ is the union of  $3^{i}$  scaled down copies of  $T_0$.

To  construct $S_1$, we let  $T_1$  be  the triangle   $T_0$ scaled by a factor of $\frac 12$ and  we consider    the translates of $T_1$  by vectors $v =(\frac 12, 0)$ and $w=(0,\,\frac 12 )$.  We  let
  $S_1=  T_1\cup(T_1+v ) \cup (T_1+w)$.   See Figure $3$.
 
  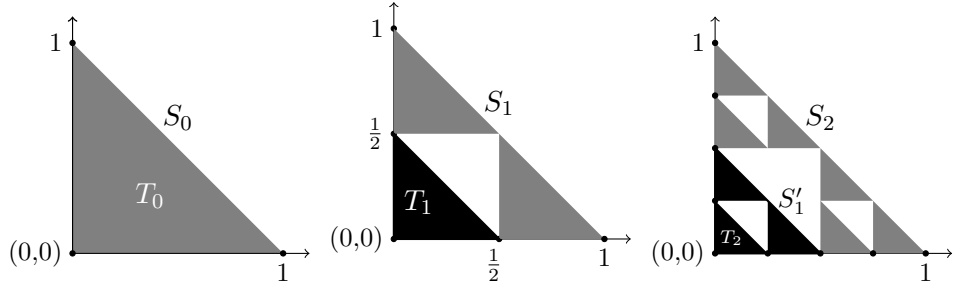
\begin{figure}
 	\hskip - .2 cm
 	\begin{tikzpicture}[scale=0.7]
 		\fill[gray] (0,0)--(4,0)--(0,4)--cycle;
 		\draw[black, ->] (0,0) -- (4.5,0);
 		\draw[black, ->] (0,0) -- (0, 4.5);
 		
 		
 		\draw[black, ->] (0,0) -- (0, 4.5);
 		
 		\draw[fill=black]  (0,0)  circle(0.05) node[black, left]{ \small{(0,0)}};
 		\draw[fill=black]  (0,4)  circle(0.05) node[black, left]{ \small{1}};
 		

 		\draw[fill=black]  (4,0)  circle(0.05) node[black, below]{ \small{1} };
 		\draw (1, 1.1) node[white, right] { {\bf $T_0$}};
 		\draw (2,  3) node[black, below] { {\bf $S_0$}};
 	\end{tikzpicture}
 	\begin{tikzpicture} [scale=0.7]
 		\draw[black, ->] (0,0) -- (4.5,0);
 		\draw[black, ->] (0,0) -- (0, 4.5);

 		\draw[fill=black]  (0,2)  circle(0.05) node[black, left]{ \small{$\frac 12$} };
 		
 		\draw[fill=black]  (0,0)  circle(0.05) node[black, left]{ \small{(0,0)}};
 		\draw[fill=black]  (0,4)  circle(0.05) node[black, left]{ \small{1}};
 		\draw[fill=black]  (2,0)  circle(0.05);
 		\draw (1.9, 0) node[black, below ]{ \small{$\frac 12$} }   ;
 		
 		\fill[gray] (2,0)--(4,0)--(2,2)--cycle; 
 		\fill[black] (0,0)--(2,0)--(0,2)--cycle;
 		
 		\fill[gray] (0,4) -- (2,2) -- (0,2) -- cycle; 
 		
 		\draw[fill=black]  (4,0)  circle(0.05) node[black, below]{ \small{1} };
 		\draw (0, .7) node[white, right] { {\bf $T_1$}};
 		\draw[black, thick] (0,0)--(2,0)--(0,2)--cycle;
 		\draw (2,  3) node[black, below] { {\bf $S_1$}};
 	\end{tikzpicture}
 	\begin{tikzpicture}[scale=0.7]
 		\draw[black, ->] (0,0) -- (4.5,0);
 		\draw[black, ->] (0,0) -- (0, 4.5);
 		\fill[gray] (0,4) -- (1,3) -- (0,3) -- cycle;

 		\fill[gray] (0,3) -- (1,2) -- (0,2) -- cycle;
 		
 		\fill[gray] (1,2 ) -- (2,2) -- (1,3) -- cycle;
 		
 		
 		\fill[gray] (2,2) -- (3,1) -- (2,1) -- cycle;
 		\fill[gray] (2,1) -- (3,0) -- (2,0) -- cycle;
 		\fill[gray] (3,0 ) -- (4,0) -- (3,1) -- cycle;
 		\fill[black] (0,0)--(2,0)--(0,2)--cycle;
 		\fill[white](1,0)--(1,1)--(0,1)--cycle;

 		\draw[fill=black]  (0,0)  circle(0.05) node[black, left]{ \small{(0,0)}};
 		\draw[fill=black]  (0,1)  circle(0.05) 
 		;
 		\draw[fill=black]  (0,2)  circle(0.05); 
 		\draw[fill=black]  (0,3)  circle(0.05);  
 		\draw[fill=black]  (1,0)  circle(0.05) 
 		;
 		\draw[fill=black]  (0,4)  circle(0.05) node[black,  left]{ {\small 1}};
 		\draw[fill=black]  (2,0)  circle(0.05); 
 		\draw[fill=black]  (3,0)  circle(0.05);  
 		\draw[fill=black]  (4,0)  circle(0.05)node[black,  below]{ {\small 1}};  
 		\draw[black, thick] (0,0)--(2,0)--(0,2)--cycle;
 		\draw (-.1 ,  .3) node[white, right] { \tiny{\bf $T_2$}};
 		\draw (2,  3) node[black, below] { {\bf $S_2$}};
 		\draw (1,  1) node[black, right] {\small {\bf $S_1'$}};
 	\end{tikzpicture}
 	\caption{ \small{   The first two steps of the construction of the Sierpinski triangle}}
 \end{figure}
To construct $S_{j+1}$ from $S_j$, we let  $S_j' $ denote the set  $S_j$ scaled by a factor $\frac 12$   and we define $S_{j+1}=  S_j'\cup(S_j'+v) \cup (S_j'+w)$.
  The Sierpinski Triangle  is  then defined as the set  $\dsize S_T := \bigcap_{j =0}^{\infty} S_j$.  
%

Every point $P$ in the Sierpinski triangle belongs to $S_i$ for all $  i\ge 0$; 
we can  denote by $ \tau_i(P)  $   the   triangle  in the set $S_i$ that contains $P$ and for every $P\in S_T$, we can construct  a  decreasing sequence of triangles $\tau(P)=\{\tau_i(P)  \})_{i \in\N}$  such that $P =\bigcap_{i=0}^{\infty} \tau_i(P)$.    The sequence $\tau(P)$ can be used to identify the position of $P$ in the Sierpinski triangle; for details,  we  refer the reader to   \cite{S} or to other books on fractals.  

\medskip

The next theorem establishes  a connection between the  Sierpinski triangle $S_T$ and the standard Egyptian fractions.   

\begin{Thm}\label{T-Main}
	Let $\bx, \by\in \ell^1(\Z_2)$ and let $h(\bx)$ and $h(\by)$ be the associated Egyptian fractions.   If
	\begin{equation}\label{e-sum-EF}
		\Sigma h(\bx+\by)= \Sigma h(\bx)+\Sigma h(\by)\end{equation}   the point   $(x,y) $ for which  $  x= [0.x_1 x_2....]_2 $ and $y=   [0.y_1y_2...]_2 $ lies in the  Sierpinski triangle $S_T  $.   
\end{Thm}
 Recall that  the  binary decimal $[0.s_1s_2...]_2$   represents   the number   $  \frac{s_1}{2}+\frac{s_2}{2^2}+...$.   
  If ${\bf s}=(s_1, s_2,\, ...)$  is in $ \ell^1(\Z)$, the decimal binary  number $[0.s_1s_2...]_2$   is finite. 
    
To prove Theorem \ref{T-Main} we  need the following

\begin{Prop}\label{P-1}
 Let  $ (x,y)\in  [0,1]\times [0,1]$.
 
a)   If $(x,y) $ lies in  the Sierpinski triangle $S_T$,     there exist  binary representations of $x$ and $y$ given by  $[0.x_1x_2...]_2$ and $[0.y_1y_2... ]_2  $  for which   $ x_j  y_j=0$  for every index $j$. 

b)  If  the  binary representations of $x$ and $y$  are finite  and satisfy $ x_j  y_j=0$ for every index $j$, then  the point $(x,y)$  lies in   $S_T$ 
 \end{Prop}

  A similar version of this proposition is known when the starting triangle is equilateral, but since we are working with a right triangle, it is helpful to provide a proof here.
  
Our proof allows us to establish part b) of Proposition \ref{P-1} only when $x$ and  $y$ have finite binary decimal expansions, which is enough to prove Theorem \ref{T-Main}. However, it is also possible to prove that part b) holds when the binary decimal expansions of $x$ and  $y$ are not finite, or in other words that the converse of part a) holds.
  
  Proposition \ref{P-1} leads to Theorem \ref{T-Main}. In fact, according to Proposition \ref{P-Sum-2}, the identity \eqref{e-sum-EF} is true if and only if  $x_jy_j = 0$ for all indices $j $. By Proposition \ref{P-1} (b), all points $ (x,y) $  in the square $ [0,1] \times [0,1] $ whose  finite binary expansions satisfy $x_jy_j = 0$ are located in the Sierpinski triangle  $ S_T $.
 
 \medskip
 To prove Proposition \ref{P-1} we need the following

\begin{Lemma}\label{L-1} 
	 For every $n\ge 0$,  we let $T_n$ be a scaled copy of the triangle $ T_0$ defined in \eqref{e-To} by a factor $\frac{1}{2^{n }}$.
	The    points  $(x,y)\in T_n$    have binary representations  $ x=  [0. x_1x_2...]_2 , \ y= [0. y_1y_2...]_2 $, with  $x_j=y_j=0$ for all  $ j\leq n$ and   $x_{n+1}y_{n+1}=0$. 
	\end{Lemma}


\begin{proof}  
		We first observe that 
	the set of  points   $x=[0.x_1x_2...]_2$  for which  $x_1=1$ is bounded below by  $\frac 12=[0.1000...]_2$  and it is bounded above by   $ 1=\frac 12 + \frac{1}{2^2 }+\frac{1}{2^3}+...   $.  Thus,   $x_1=1$  when $x\in [\frac12,\, 1)$, and       $x_1=0 $ when  $ x\in [0,\,\frac 12 )$.  
Since  the points   in the  initial triangle $T_0$  satisfy  $0\leq x+y\leq 1$,  we can  easily verify  that  $P=(\frac 12,\,\frac 12)$ is
   the only point   in  $T_0$ for which $x_1=y_1=1$.  However,   we can also write $\frac 12=[0.01111...]_2$, and so  also the components of $P$ satisfy     $x_1y_1=0$.

Since dividing  by powers of $2$ shifts the decimal point in  the binary representation in base two,   we can verify by induction that  the first $n$ decimal digits of  points in the interval  $[0, \frac{1}{2^{n  }})$ are zero, and    the  triangle  $T_{n }$ with vertices   $(0,0), \ (\frac 1{2^{n }}, 0),\ (0, \frac{1}{2^{n  }})$  contains   points  $(x,y)$  whose binary representation  satisfies $x_j=y_j=0$ for all  $1\leq j\leq n$ and  $x_{n+1}y_{n+1}=0$.
\end{proof}

\begin{proof}[Proof of Proposition \ref{P-1}]
 We recall that   $S_T=\bigcap_{n=0}^\infty S_n$,  where $S_0$ is the initial triangle $T_0$  and each   $S_j$ is the union of   re-scaled and translated copies of $ S_{j-1}$.  
 
  We prove by induction that   the binary decimal expansions of points $(x,y)\in S_n$  satisfy $x_jy_j=0$ for every $j\leq n+1$. Furthermore, we also prove that  $S_n$ contains all points  $(x,y)$ for which $x_jy_j=0$ for every $j\leq n+1$ and $x_j=y_j=0$ whenever $j\ge n+2$.
 
  We have proved in Lemma \ref{L-1} that points   in  the initial triangle $T_0 $  have a decimal representations that satisfy $x_1y_1=0$.     The only points $(x,y)$  for which $x_j= y_j=0$ whenever $j\ge 2$  and $x_1y_1=0$ are  $ (0,0)$, $(\frac 12, 0)$ and $(0, \frac 12)$, and these points belong to $T_0$.

 Assume that  the decimal expansion of points $(x,y) $ in $  S_n$  satisfies   $x_jy_j=0$ for every $j\leq n+1$; assume also that $S_n$ contains  all points  $(x,y)$ for which $x_j=y_j=0$ whenever $j\ge n+2$  and $x_jy_j=0$ for every $j\leq n+1$.
 
 Let  $\t S_n$ represent  the scaled copy of $S_n$  by a factor  $\frac 12$.  This scaling shifts the binary decimal point of each coordinate in $S_n$, and so    for any $(x,y)$  in $\t S_n$  we have $x_1=y_1=0$  and   $x_jy_j=0 $ for $j=2,..,\, n+1$. 
  
By definition,  $S_{n+1}= \t S_n \cup+[(0, \frac 12)+ \t S_n ] \cup [(\frac 12, 0)+\t S_n]$.
 Adding  $ \frac 12$  only  affects the  first digits of the binary decimal expansion of the coordinates, and so  for every  points  $(x,y)\in S_{n+1}$  and every $j=1,... ,\, n+2$, we have 
$x_jy_j=0  $. Furthermore, our construction shows that $S_{n+1}$ contains all points  $(x,y)\in [0,1]\times[0,1]$ for which $x_j=y_j=0$ whenever $j\ge n+3$.

 The Sierpinski triangle $S_T$ is the intersection  of the  $S_n$, and so  all points  $(x,y)$ in $S_T$  have  binary decimal expansions that satisfy $x_jy_j=0$ for every index $j   $.    Our proof  also shows that  if the decimal binary representation of points  $(x,y)\in [0,1]\times[0,1]$  are finite and satisfy  $x_jy_j =0$  for every index $j$, then $(x,y)$ lie  in $S_T$.
   \end{proof}

\subsection{The  hexagon  snowflake fractal  }

 In this section we construct  the fractal in figure 4 
 and  we show how  it can be used to characterize the sum of signed Egyptian fractions.   
 
\begin{figure}[h] 
		\includegraphics[width=0.4\textwidth]{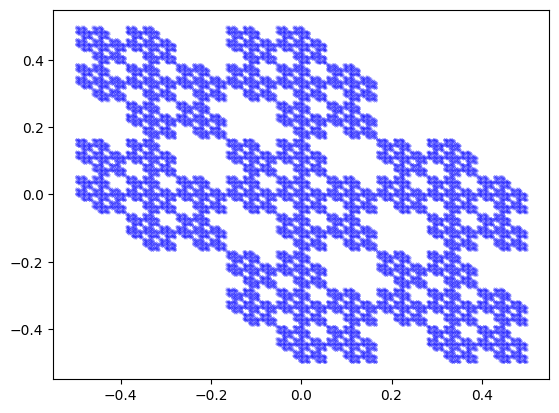}
		\caption{\small{  The  hexagon snowflake fractal $S_F$}}
		\label{fig:snowflake}
	\end{figure}

 The {\it  hexagon snowflake fractal} is formally defined   as the intersection of a decreasing sequence of sets $\{G_i\}_{i = 0}^{\infty}$ defined recursively as follows.

 $G_0$    is  the  gray polygon  $H_0$  in   figure 5.
%
  $G_1$  consists  7   copies of     $H_0$  scaled by  a factor $\frac 13$  and translated by vectors $v_0=(0,0)$, $v_1=( \frac 1{3},\, 0)$, $v_2=( 0,\, \frac 1{3} )$, 
 $v_3=( \frac 13,\, -\frac 1{3} )$, $v_4=-v_1$, $v_5=-v_2$ and $v_6=-v_3$, and
each subsequent $G_j$ is  a union of $7 $   scaled-down copies of $G_{j-1}$   translated by vectors $v_0,\, ...,\, v_6$.
We let $ S_F  := \cap_{n=0}^\infty G_n$.

	
%
	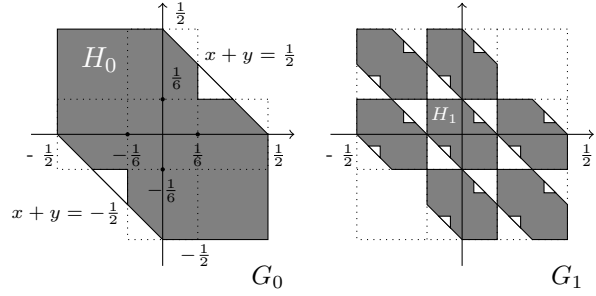
\begin{figure}
		\begin{tikzpicture}[scale=0.7]
			\draw [ fill=gray]  (2,0)--(0,2) --  (-2,2) -- (-2,  0) --(0, -2)--(2,-2)-- cycle;
			\draw[black, ->] (- 2.5,0) -- (2.5,0);
			\draw[black, ->]  (0,-2.5 ) -- (0, 2.5);
			\draw [ fill=white]  (4/3, 2/3)--(2/3,2/3) --  (2/3, 4/3 )  -- cycle;
			\draw [ fill=white]  (-4/3, -2/3)--(-2/3,-2/3) --  (-2/3, -4/3 )  -- cycle;
			
			\draw ( 2,  -2.3) node[black, below] { \small{\bf $G_0$}};
			\draw ( -1.2 ,  1) node[white, above] { {\bf $H_0$}};
			\draw (2.2,  0) node[black, below] {\tiny{\bf $ \frac 12$}};
			\draw (0,2.3) node[black, right] {\tiny{\bf $ \frac 12$}};
			\draw (-2.3,  0) node[black, below] {\tiny{- $\frac 12$}};
			\draw (0,-2.3 ) node[black, right] {\tiny{ $-\frac 12$}};
			
			\draw[black, dotted] (2/3, 2)--(2/3,-2)--(-2/3,-2)--(-2/3, 2)--cycle;
			\draw[black, dotted] (2, 2/3 )--(-2, 2/3 )--(-2, -2/3)--(2, -2/3)--cycle;
			 
			\draw[fill=black]  (-2/3,0)  circle(0.03) node[black, below]{ \tiny{$-\frac 16$} }  ;
			\draw[fill=black]  ( 2/3,0)  circle(0.03) node[black, below]{ \tiny{$ \frac 16$} }  ;
			\draw[fill=black]  (0,2/3 )  circle(0.03); 
			\draw ( .2 ,  2/3) node[black, above] {\tiny {  $\frac 16$}};
			\draw[fill=black]  (0,-2/3 )  circle(0.03  ) node[black, below]{ \tiny{$-\frac 16$} }  ;
			\draw (.6 , 3/2) node [black, right] 	{ \tiny{$x+y=\frac 12$} };
			\draw (-.6 , -3/2) node [black, left] 	{ \tiny{$x+y=-\frac 12$} };
		\end{tikzpicture}  
	\begin{tikzpicture}[scale=0.7]
		\draw [black, thin, dotted]
		  (2,2) -- (-2,2) -- (-2,-2 ) -- (2,  -2) -- cycle;

		\draw [ fill=gray]  (2/3,0)--(0, 2/3) --  (-2 /3,2/3) -- (-2/3,  0) --(0, -2/3)--(2/3,-2/3)-- cycle;
		\draw [ fill=white]  (4/9, 2/9)--(2/9,2/9) --  (2/9, 4/9 )  -- cycle;
		\draw [ fill=white]  (-4/9, -2/9)--(-2/9,-2/9) --  (-2/9, -4/9 )  -- cycle;

\draw [ fill=gray]  (6/3,0)--(4/3, 2/3) --  ( 2 /3,2/3) -- ( 2/3,  0) --(4/3, -2/3)--(6/3,-2/3)-- cycle;
\draw [ fill=white]  (16/9, 2/9)--(14/9,2/9) --  (14/9, 4/9 )  -- cycle;
\draw [ fill=white]  (8/9, -2/9)--(10/9,-2/9) --  (10/9, -4/9 )  -- cycle;

\draw [ fill=gray]  (-2/3,0)--(-4/3, 2/3) --  (-6 /3,2/3) -- (-6/3,  0) --(-4/3 , -2/3)--(-2/3,-2/3)-- cycle;
\draw [ fill=white]  (-8/9, 2/9)--(-10/9,2/9) --  (-10/9, 4/9 )  -- cycle;
\draw [ fill=white]  (-16/9, -2/9)--(-14/9,-2/9) --  (-14/9, -4/9 )  -- cycle;

\draw [ fill=gray]  (-2/3,4/3)--(-4/3, 6/3) --  (-6 /3,6/3) -- (-6/3,  4/3) --(-4/3 ,  2/3)--(-2/3, 2/3)-- cycle;
\draw [ fill=white]  (-8/9, 14/9)--(-10/9,14/9) --  (-10/9, 16/9 )  -- cycle;
\draw [ fill=white]  (-16/9, 10/9)--(-14/9,10/9) --  (-14/9, 8/9 )  -- cycle;

\draw [ fill=gray]  (2/3,4/3)--(0, 6/3) --  (-2 /3,6/3) -- (-2/3,  4/3) --(0,  2/3)--(2/3, 2/3)-- cycle;
 \draw [ fill=white]  (4/9, 14/9)--(2/9,14/9) --  (2/9, 16/9 )  -- cycle;
\draw [ fill=white]  (-4/9, 10/9)--(-2/9,10/9) --  (-2/9, 8/9 )  -- cycle;
		 
		 \draw [ fill=gray]  (2/3,-4/3)--(0, -2/3) --  (-2 /3,-2/3) -- (-2/3,  -4/3) --(0, -6/3)--(2/3,-6/3)-- cycle;
		   \draw [ fill=white]  (4/9, -10/9)--(2/9,-10/9) --  (2/9, -8/9 )  -- cycle;
		 \draw [ fill=white]  (-4/9, -14/9)--(-2/9,-14/9) --  (-2/9, -16/9 )  -- cycle;
		 
		 \draw [ fill=gray]  (6/3,-4/3)--(4/3, -2/3) --  ( 2 /3,-2/3) -- ( 2/3,  -4/3) --(4/3, -6/3)--(6/3,-6/3)-- cycle;
		  \draw [ fill=white]  (16/9, -10/9)--(14/9,-10/9) --  (14/9, -8/9 )  -- cycle;
		 \draw [ fill=white]  (8/9, -14/9)--(10/9,-14/9) --  (10/9, -16/9 )  -- cycle;
		 
		\draw[black, ->] (- 2.5,0) -- (2.5,0);
		\draw[black, ->]  (0,-2.5 ) -- (0, 2.5);
				\draw (-2.3,  0) node[black, below] {\tiny{- $\frac 12$}};
				\draw ( 2.3,  0) node[black, below] {\tiny{ \bf $\frac 12$}};
				\draw (-.4,  .7 ) node[white, below] {\tiny{ {\bf $H_1$}}};
				\draw[black, dotted] (2/3, 2)--(2/3,-2);
					\draw[black, dotted] (-2/3, 2)--(-2/3,-2);
						\draw[black, dotted] (2, 2/3 )--(-2, 2/3 );
						 \draw[black, dotted] (2,-2/3 )--(-2,-2/3 );
						\draw ( 2,  -2.3) node[black, below] { \small{\bf $G_1$}};
	\end{tikzpicture} \caption{The first two steps of the construction of the hexagon snowflake fractal $S_F$.}

\end{figure}

\medskip
The following is  an    analog of Theorem \ref{T-Main} for signed Egyptian fractions. 
 \begin{Thm}\label{T-Main-2}
 	Let $\bx, \by\in \ell^1(\Z_3)$ and let $ h  (\bx)$ and $ h  (\by)$ be the associated signed Egyptian fractions.  
 	If \begin{equation}\label{e-rel-3}\Sigma  h  (\bx+\by)= \Sigma  h  (\bx)+\Sigma  h  (\by)  \end{equation}
 	and ${\bf z(x,y)}=0$, 
  then  the  point $ (x,y)$ with  $x=[   0.x_1 x_2...]_3, \   y=[ 0.y_1y_2...]_3$ lies in the  fractal snowflake $S_F$. 

 \end{Thm}
 Recall that    ${\bf z} ={\bf z(x, y)}$ is the vector whose components $z_j$ are equal to  $1$  if $x_j=y_j=1$,  to $-1$ if  $x_j=y_j=-1$  and to  zero in all other cases.

 The proof of Theorem \ref{T-Main-2} is based on  Proposition \ref{P-2} below.

 \begin{Prop}\label{P-2} Let  $ (x,y)\in  [-\frac 12, \frac 12]\times [\frac 12, \frac 12]$.
 	
 	
 	a)  If $(x,y) $ lies in the  hexagon snowflake  fractal  $S_F$,  there exist    ternary decimal  representations of $x$ and $y$     given by  $[0.x_1x_2...]_3$ and $[0.y_1y_2... ]_3  $  for which   either $ x_j  y_j=0$ or $x_j+y_j=0$ for every index $j$. 
 	
 	b)  If  the  ternary decimal representations of $x$ and $y$  are finite  and satisfy either $ x_j  y_j=0$ or $x_j+y_j=0$ for every index $j$, then  $(x,y)$  lies in   $S_F$ 
  
 \end{Prop}
  

  According to Proposition \ref{P-Sum-3} (b), the identity \eqref{e-rel-3} holds whenever $ {\bf z} = 0 $; as noted in Section 2, $ {\bf z} = 0 $ if and only if  the components of $\bx$ and $\by$ satisfy either $ x_j y_j = 0 $ or $ x_j + y_j = 0 $ for every index $j$.  By Proposition \ref{P-2} (b), all points $ (x, y) $ whose decimal ternary expansions are finite and satisfy either $ x_j y_j = 0 $ or $ x_j + y_j = 0 $ for every index $j$ are contained in the hexagon  snowflake fractal $ S_F$. Therefore, Theorem \ref{T-Main-2} is proved.
  
  \medskip
  To prove Proposition \ref{P-2} we need the following 
 
 \begin{Lemma}\label{L-2} 
 	For every $n\ge 0$,  we let $H_n$ be a scaled copy of the polygon $ H_0$  by a factor $\frac{1}{3^{n }}$.
 	The    points  $(x,y)\in H_n$    have balanced ternary   representation  $ x=  [0. x_1x_2...]_3 , \ y= [0. y_1y_2...]_3 $  for which   $x_j=y_j=0$ for all  $  j\leq n$,  and  either $x_{n+1}y_{n+1}=0$ or  $x_{n+1}+y_{n+1}=0$.  
 \end{Lemma}
 
 \begin{proof}   
 	We begin by noting  that   the set  of points  $x= [0. x_1x_2... ]_3$     for which $x_1=0$ is bounded above by $  \frac 1{3^2}+ \frac {1}{3^3}+ ... = \frac 16
 	$ and it is bounded below by $  -\frac 1{3^2}-\frac {1}{3^3}- ... = -\frac 16
 	$. Thus, the  points in the interval $[-\frac 16, \frac 16]$  are the only points in   $[-\frac 12, \,\frac 12]$  that have a  balanced ternary representation    that starts with $x_1=0$.  
 	 The points in the intervals  $(\frac 16, \frac 12]$     and $[-\frac 12, -\frac 16)$ can be written as  $[0.1 x_2 x_3...]_3$ and $[0.\ov 1 x_2 x_3...]_3$, resp. Note that $\frac 16$  and $-\frac  16$ can also be represented as $[0.1\ov 1\,\ov 1\,...  ]_3$ and $[0.\ov 1 \,1  1...  ]_3$, resp.

 	For the points $(x,y)$ in the squares $[\frac 16, \frac 12] \times[-\frac 12,\,-\frac 16 ] $  (resp. $[ - \frac 12,\,-\frac 16 ] \times[\frac 16,\,  \frac 12 ]$) we have     $x_1=1$ and $y_1=-1$ (resp. $x_1=-1$ and $y_1=1$) and so $x_1+y_1=0$.  
 	   For  the points  in the rectangle  $(-\frac 16, \frac 16)\times[-\frac 12, \frac 12]$, we have $x_1=0$, so $x_1y_1=0$.  Similarly, $y_1=0$ for the points in the rectangle $[-\frac 12, \frac 12]\times (-\frac 16, \frac 16) $
 	 	
  Based on these observations, we can conclude that 	the polygon $ H_0$   consists of points $(x,y)\in [-\frac 12, \frac 12]\times[-\frac 12, \frac 12] $ for which  either $  x_1 y_1=0$   or $x_1+y_1=0$.

 	Since dividing  by powers of $3$ shifts the decimal point in  the   balanced ternary representation, 
 	 we can verify by induction that  points in  the polygon   $H_{n }$ satisfy   Lemma \ref{L-2}. 
 \end{proof}
 
 The proof of  Proposition \ref{P-2}  is similar to that of Proposition \ref{P-1} and   it is omitted.

 \end{document}